\documentclass[a4paper,12pt]{article}
\usepackage{latexsym}
\usepackage{amsmath}
\usepackage{amssymb}
\usepackage{enumerate}
\usepackage{bm}
\usepackage{mathrsfs}
\usepackage{color}

\usepackage[anchorcolor=blue,%
 bookmarks=true,%
 bookmarksnumbered=true,%
 colorlinks=true,%
 citecolor=cyan,%
 linkcolor=blue,%
 dvipdfmx%
]{hyperref}

\definecolor{refkey}{gray}{0.5}
\definecolor{labelkey}{gray}{0.2}

\usepackage{theorem}
\newtheorem{theorem}{Theorem}[section]
\newtheorem{proposition}[theorem]{Proposition}
\newtheorem{lemma}[theorem]{Lemma}
\newtheorem{corollary}[theorem]{Corollary}

\theorembodyfont{\rmfamily}
\newtheorem{proof}{\textmd{\textit{Proof.}}}

\newtheorem{remark}[theorem]{Remark}

\newtheorem{definition}[theorem]{Definition}

\makeatletter

\@addtoreset{equation}{section}
\makeatother

\newcommand{\qedd}{\hfill \Box}
\newcommand{\ve}{\varepsilon}

\newcommand{\e}{\mathrm{e}}

\newcommand{\R}{\ensuremath{\mathbb{R}}}

\newcommand{\cC}{\ensuremath{\mathcal{C}}}

\newcommand{\cI}{\ensuremath{\mathcal{I}}}

\newcommand{\fm}{\ensuremath{\mathfrak{m}}}

\def\Ric{\mathop{\mathrm{Ric}}\nolimits}

\def\RCD{\mathop{\mathrm{RCD}}\nolimits}

\title{A rigidity result of spectral gap on Finsler manifolds and its application}
\author{Cong Hung MAI\thanks{
Department of Mathematics, Osaka University, Osaka 565-0871, Japan
({\sf hungmcuet@gmail.com})}}

\date{\today}
\begin{document}

\maketitle

\begin{abstract}
We investigate the rigidity problem for the sharp spectral gap on Finsler manifolds of weighted Ricci curvature bound  $\textup{Ric}_{\infty} \geq K > 0$. Our main results show that if the equality holds, the manifold necessarily admits a diffeomorphic splitting (or isometric splitting in the particular class of Berwald spaces). This splitting phenomenon is comparable to the Cheeger-Gromoll type splitting theorem by Ohta. We also obtain the rigidity results of logarithmic Sobolev and Bakry--Ledoux isoperimetric inequalities via needle decomposition as corollaries.
\end{abstract}


\section{Introduction}
In geometric analysis, the rigidity problem of a geometric inequality is an essential subject as it would characterize the model spaces attaining the equality case. Obata rigidity theorem \cite{Ob} - which studies the rigidity of the spectral gap under nonnegative Ricci curvature bound, is one of the most popular rigidity results with the constraint of Ricci curvature bound. It stated that if the first nonzero eigenvalue of the Laplacian for a Riemannian manifold attains the global lower bound (concerning the lower bound of Ricci curvature), the manifold must be isometric to a sphere.

In the last decade, this rigidity problem has been generalized to the setting of weighted Riemannian manifolds. In this setting, the weighted Ricci curvature bound $\textup{Ric}_{N} \geq K$ is equivalent to the \emph{curvature-dimension} condition $\textup{CD}(K,N)$ by Lott, Sturm, and Villani. Here the parameters $K$ and $N$ are usually regarded as ``a lower bound of the Ricci curvature'' and ``an upper bound of the dimension.'' The counterparts of Obata's rigidity theorem on weighted manifolds of Ricci curvature bound $\textup{Ric}_{N} \geq K > 0$ were studied in \cite{Ke,Ku} ($N > 0$), \cite{CZ}($N = \infty$) and \cite{Ma1} ($N < -1$). Interestingly, the model spaces in the last two cases are non-compact (product spaces including Gaussian and hyperbolic spaces) even when the curvature is positive. It is natural to consider the generalization of these studies into more general classes of nonsmooth spaces satisfying the curvature-dimension condition. On $\textup{RCD}(K,\infty)$-spaces, the rigidity of the sharp spectral gap is studied on \cite{GKKO}. The result implies a similar isometric splitting to the case of weighted Riemannian manifolds. 

The article investigates the rigidity problem of the spectral gap on $\textup{CD}(K,\infty)$-Finsler manifolds. Our main theorem implies a diffeomorphic splitting phenomenon of Gaussian spaces. This result can be compared to the Cheeger--Gromoll splitting in \cite{Osplit} as the eigenfunction associated with the optimal first nonzero eigenvalue behaves similarly to the Busemann function. 

The article is organized as follows: In Section 2, we briefly introduce Finsler manifolds, including the comparison geometry of weighted Ricci curvature. Section 3 considers the rigidity problem on $\textup{CD}(K,\infty)$-Finsler manifolds and Berwald spaces. We study the rigidity of logarithmic Sobolev and Bakry--Ledoux isoperimetric inequalities as corollaries in Section 4.

\textit{Acknowledgements}.
The author wants to thank Professor Shin-ichi Ohta for valuable advice. The author was supported by JSPS Grant-in-Aid for Scientific Research (KAKENHI) 20J11328 and 19H01786.

\section{Preliminaries}
This section provides a brief introduction to Finsler geometry. The author recommends the readers to read these books \cite{BCS, Obook} for fundamental concepts of Finsler manifolds.
\subsection{Finsler geometry}

A Finsler manifold $(M,F)$ is a pair of a $C^{\infty}$-manifold $M$ and a Finsler structure $F$. On a local coordinate $(x^{i})_{i=1}^{n}$ of an open set $U \subset M$, we use a fiber-wise linear coordinate $(x^{i},v^{j})_{i,j=1}^{n}$ of the tangent bundle $TU$ such that 

\[ v = \sum_{j=1}^{n} v^{j}\frac{\partial}{\partial x^{j}}\bigg|_{x} \in T_{x}M\]
for $x \in U$.

\begin{definition}
A nonnegative function $F: TM \rightarrow [0,\infty)$ is a $\mathcal{C}^{\infty}$-$\emph{Finsler structure}$ if it satisfies the following conditions:

\begin{enumerate}[(a)]
    \item $F$ is  $\mathcal{C}^{\infty}$ on $TM\setminus 0$ (Regularity).
    \item $F(cv) = cF(v)$ for all $v\in TM$ and $c>0$ (Positive 1-homogeneity).
    \item The $n$-square matrix 
    \begin{equation}
        (g_{ij}(v))_{i,j=1}^{n} := \frac{1}{2}\Bigg(  \frac{\partial^{2}[F^2]}{\partial v^i \partial v^j} \Bigg)_{i,j=1}^{n}
    \end{equation}
    is positive-definite for any tangent vector $v\in TM\setminus 0$ (Strong convexity).
\end{enumerate}
\end{definition}
We note that in general $F(v) \neq F(-v)$. If $F(v) = F(-v)$ for all $v\in TM\setminus 0$, we say that $(M,F)$ is \emph{reversible}. On a non-reversible Finsler manifold, we define the \emph{reverse Finsler structure} $\overleftarrow{F}$ of $F$ by putting $\overleftarrow{F} (v) := F(-v)$ for $v\in TM$. 

On a Finsler manifold, we define the asymmetric distance from $x$ to $y$ as follows:
\[ d(x,y) := \inf_{\eta} \int_{0}^{1} F(\dot{\eta}(t))dt,\]
where $\eta: [0,1] \rightarrow M$ runs over all $\mathcal{C}^{1}$-curves with $\eta(0)=x$ and $\eta(1)=y$. We note that $d(x,y) \ne d(y,x)$ in general. We say that a $\mathcal{C}^{\infty}$-curve $\eta$ is a geodesic if it is locally minimizing and has a constant speed with respect to distance $d$. If there is a geodesic $\eta:[0,1] \rightarrow M$ with $\dot{\eta}(0) = v$ for $v\in T_{x}M$, the \emph{exponential map} is defined as $\text{exp}_{x}(v):=\eta(1)$. $(M,F)$ is called \emph{forward complete} if the exponential map is defined everywhere on $TM$. If the reverse structure $(M,\overleftarrow{F})$ is forward complete then $(M,F)$ is \emph{backward complete}. We note that the forward completeness is not necessarily equivalent to the backward completeness when the \emph{reversibility constant} 
\[ \Lambda_{F} := \sup_{v\in TM\setminus 0} \frac{F(-v)}{F(v)} \in [1,\infty]\]
is infinite. We call $u\in\mathcal{C}^{1}(M)$ a \emph{L-Lipschitz} function if 
\[ f(y) -f(x) \leq Ld(x,y)\]
holds for all $x$, $y\in M$. We note that $| f(y) -f(x) |\leq Ld(x,y)$ does not hold on non-reversible Finsler manifolds.

For $v\in T_{x}M$, $(g_{ij}(v))_{i,j=1}^{n}$ induces a (coordinate free) Riemannian structure
\[ g_{v} \bigg(  \sum_{i=1}^{n}a_{i}\frac{\partial}{\partial x^{i}} \bigg|_{x}, \sum_{j=1}^{n}b_{j}\frac{\partial}{\partial x^{j}}\bigg|_{x}\bigg) =  \sum_{i,j=1}^{n} g_{ij}(v)a_{i}b_{j}.\]
We have $g_{v}(v,v) = F^{2}(v)$. $(T_{x}M,g_{v})$ is the best Riemannian approximation of $(T_{x}M,F)$ in the direction $v$. 
Next, we define the dual Minkowski norm $F^{*}: T^{*}M \rightarrow [0,\infty)$ acting on $\alpha = \sum_{i=1}^{n} \alpha_{i}dx^{i} \in T^{*}M$:
\[ F^{*}(\alpha) := \sup_{v\in T_{x}M, F(v) \leq 1} \alpha(v) = \sup_{v\in T_{x}M, F(v) = 1} \alpha(v),\]
and the dual structure in the coordinates $(x^{i},\alpha_{j})_{i,j=1}^{n}$ of $T^{*}U$:

\[ g_{ij}^{*}(\alpha) := \frac{1}{2} \frac{\partial^{2}[(F^{*})^{2}]}{\partial \alpha_{i}\partial \alpha_{j}}(\alpha), \qquad i,j = 1,2,\ldots,n,\]
for $\alpha \in T^{*}U\setminus 0$. Let $\mathcal{L}^{*}: T^{*}M \rightarrow TM$ denote the Legendre transform sending $\alpha \in T^{*}M$ to the unique tangent vector $v\in T_{x}M$ such that $F(v)=F^{*}(\alpha)$ and $\alpha(v) = F^{*}(\alpha)^2$. In coordinate,
\[\mathcal{L}^{*}(\alpha) =  \sum_{i,j=1}^{n} g_{ij}^{*}(\alpha)\alpha_{i}\frac{\partial}{\partial x^{j}} \bigg|_{x} =  \sum_{j=1}^{n}\frac{1}{2} \frac{\partial[(F^{*})^{2}]}{\partial \alpha_{j}}(\alpha)\frac{\partial}{\partial x^{j}} \bigg|_{x}\]
for $\alpha\in T^{*}M\setminus 0$.

The \emph{covariant derivative} of a $\mathcal{C}^{1}$-vector field $X$ by $v\in T_{x}M$ with reference vector $w\in T_{x}M\setminus 0$ is defined as follows:

\begin{equation}
    D_{v}^{w}X(x) := \sum_{i,j=1}^{n}\bigg\{ v^{j}\frac{\partial X^{i}}{\partial x^{j}}(x) + \sum_{k=1}^{n}\Gamma^{i}_{jk}(w)v^{j}X^{k}(x) \bigg\} \frac{\partial}{\partial x^{i}} \bigg|_{x} \in T_{x}M, 
\end{equation}
where $\Gamma_{jk}^{i}$ is the Chern connection coefficient.
For a $\mathcal{C}^{1}$-curve $\eta:[0,l]\rightarrow M$, $\eta$ is geodesic if and only if it satisfies the \emph{geodesic equation}
\[ D_{\dot{\eta}}^{\dot{\eta}}\dot{\eta} = 0.\]

We close this subsection by introducing a particular class of Finsler manifolds with affine property.
\begin{definition}
We call a Finsler manifold $(M,F)$ a \emph{Berwald space} if $\Gamma_{jk}^{i}$ are constant on $T_{x}M\setminus{0}$ for all $x\in M$.
\end{definition}
On Berwald spaces, the covariant derivative does not depend on the reference vector. Moreover, Berwald spaces have finite reversibility and hence the forward and backward completeness are equivalent (see Section 6.3 in \cite{Obook}).

\subsection{Weighted Finsler manifolds}
In this article, we consider the setting of weighted Finsler manifolds. A weighted Finsler manifold is a triple $(M,F,\fm)$ of a boundaryless Finsler manifold $(M,F)$ equipped with a $\mathcal{C}^\infty$-Borel measure $\fm$.

For a differentiable function $u: M\rightarrow \mathbb{R}$, we define the \emph{gradient vector} of $u$ by putting $\nabla u(x):=\mathcal{L}^{*}(du(x))$. The expression of $\nabla u$ in coordinate can be written as

\[ \nabla u = \sum_{i,j=1}^{n} g^{*}_{ij}(du)\frac{\partial u}{\partial x^{j}}\frac{\partial}{\partial x^{i}}\]
on the set $M_{u} := \{x\in M| du(x)\ne 0\}$. $M_u$ is the called the \emph{essential domain} of $u$ as the Legendre transformation is only continuous at the zero section and $g^{*}_{ij}(du)$ is not defined outside $M_u$. On $M\setminus M_u$, we put $\nabla u = 0$. We note that $F(\nabla u) = F^{*}(du)$ by definition.

For a differentiable vector field $V$ on $M$ and $F(V)\in L^{1}_{\text{loc}}(M)$, we define the \emph{divergence} $\textup{div}_{\fm}$ with respect to $\fm$ by

\[ \int_{M}\phi\mathrm{div}_{\fm}(V)d\fm := -\int_{M}d\phi(V)d\fm \qquad \forall \phi \in \mathcal{C}_{c}^{\infty}(M).\]

We denote by $H_{\textup{loc}}^{1}(M)$ the space of weakly differentiable functions $u$ on $M$ such that both $u$ and $F^{*}(du)$ belong to $L_{\textup{loc}}^{2}(M)$. 

\begin{definition}[Nonlinear Laplacian]
The \emph{nonlinear Laplacian} of a function $u \in H_{\textup{loc}}^{1}(M)$ is defined by $\Delta u := \textup{div}_{\fm}(\nabla u)$ in the weak sense that for all $\phi\in\cC_{c}^{\infty}(M)$, we have
\begin{equation}
\int_{M} \phi\Delta u d\fm := - \int_{M} d\phi (\nabla u)d\fm.
\end{equation}
\end{definition}

Nonlinear heat equation $\partial_{t}u = \Delta u$ on Finsler manifolds was intensively studied in (\cite{GS, OS1,OS2,OS3}). We define the \emph{energy functional} $\mathcal{E}: H_{\textup{loc}}^{1}(M) \rightarrow [0,\infty]$ by
\[ \mathcal{E}(u) := \frac{1}{2} \int_{M} F^{*}(du)^2 d\fm.\]
The \emph{Sobolev space} associated with $\mathcal{E}$ is defined by
\[ H^{1}(M) := \{u\in L^{2}(M) \cap H_{\textup{loc}}^{1}(M)| \mathcal{E}(u) + \mathcal{E}(-u) < \infty \}.\]
We denote by $H^{1}_{c}(M)$ the spaces of functions in $H^{1}(M)$ with compact support and by $H^{1}_{0}(M)$ the closure of $\mathcal{C}_{c}^{\infty}(M)$ with respect to the \emph{Sobolev norm}
\[ ||u||_{H^{1}(M)} := \sqrt{||u||_{L^2(M)}+ \mathcal{E}(u) + \mathcal{E}(-u)},\]
where $\mathcal{C}_{c}^{\infty}(M)$ is the set of smooth function with compact support. Lemma 11.4 in \cite{Obook} implies that if $(M,F)$ has finite reversibility $\Lambda_{F} < \infty$ and is complete, then $H_{0}^{1}(M) = H^{1}(M)$ and constant functions on $M$ belong to $H_{0}^{1}(M)$ provided $\fm(M)<\infty$. 

We will introduce preliminary results of comparison geometry on weighted Finsler manifolds of positive weighted Ricci curvature in the following parts. We begin with the definition of the weighted Ricci curvature.
\begin{definition}\label{def:Ric}
Given a unit vector $v\in T_{x}M$, let $V$ be a $C^{\infty}$-vector field on a neighborhood $U$ of $x$ such that all integral curves of $V$ are geodesic. Let $\Psi\in \mathcal{C}^{\infty}(U)$ be the density of $\fm$ when we decompose it as $\fm = \textup{e}^{-\Psi}\text{vol}_{g_{V}}$ on $U$, and $\text{vol}_{g_{V}}$ is the standard volume form of $g_{V}$. Let $\eta:(-\epsilon,\epsilon)\rightarrow M$ be the geodesic satisfying $\dot{\eta}(0) =v$. Then, for $N\in(-\infty,0)\cap(n,\infty)$, we define the \emph{weighted Ricci curvature} as follows:
\[ \textup{Ric}_{N}(v) := \textup{Ric}(v) + (\Psi\circ\eta)''(0) - \frac{ (\Psi\circ\eta)'(0)^{2}}{N-n}.\]
We also define as the limit:
\[ \textup{Ric}_{\infty}(v) := \textup{Ric}(v) + (\Psi\circ\eta)''(0). \]
For $c\geq 0$, we set $\text{Ric}_{N}(cv) := c^{2}\textup{Ric}_{N}(v)$ by convention.
\end{definition}
If $\textup{Ric}_{N}(v) \geq KF^{2}(v)$ for all $v\in TM$, we write $\textup{Ric}_{N} \geq K$. This curvature bound is equivalent to the \emph{curvature-dimension condition} CD($K,N$) (\cite{Oneg}). The condition $\text{Ric}_{\infty} \geq K > 0$ implies that $(M,F,\fm)$ enjoys the Gaussian decay and hence its total mass is finite by Theorem 9.19 in \cite{Obook}. 

In the following paragraphs, we will introduce a vital estimate related to weighted Ricci curvature - the \emph{Bochner formula}. We give some basic notation first. For a twice differentiable function $u$ and $x\in M_u$, we define the \emph{Hessian} $\nabla^{2}(u)(x)\in T^{*}_{x}M \otimes T^{*}_{x}M$ via the covariant derivative:

\[ \nabla^{2}(u)(v) := D_{v}^{\nabla u}(\nabla u)(x) \in T_{x}M, \qquad v\in T_{x}M.\]
Note that $g_{\nabla u}(\nabla^{2}(u)(v),w) = g_{\nabla u}(v,\nabla^{2}(u)(w))$ for all $v,w\in T_{x}M$ and $x\in M_{u}$ (see \cite{OS3}).

Let $f\in H^{1}_{\text{loc}}(M)$. For a measurable vector field $V$ which is nonzero almost everywhere on $M_{f}$, we define the gradient vector field and linearized Laplacian on the approximated weighted Riemannian structure $(M,g_{V},\fm)$ as follows:

\[\nabla^{V}f :=\begin{cases} \sum_{i,j=1}^{n}g^{ij}(V)\frac{\partial f}{\partial x^{j}}\frac{\partial}{\partial x^{i}}
 & \text{on}\ M_{f}, \\
 0 & \text{on}\ M\setminus M_{f}, \end{cases} \qquad \Delta^{V}f:= \text{div}_{\fm}(\nabla^{V}f),\]
where the Laplacian is defined in terms of distribution and $(g^{ij}(v))_{i,j=1}^{n}$ denotes the inverse matrix of$(g_{ij}(v))_{i,j=1}^{n}$. We have:
\begin{equation}
    df_{2}(\nabla^{V}f_{1}) = g_{V}(\nabla^{V}f_{1},\nabla^{V}f_{2}) = df_{1}(\nabla^{V}f_{2}) 
\end{equation}
for $f_{1}, f_{2}\in  H^{1}_{\text{loc}}(M)$ and $V$ such that $V\neq 0$ almost everywhere on $M$. It is known in \cite{OS1} that $\nabla^{\nabla u}u = \nabla u$ and $\Delta^{\nabla u}u = \Delta u$ for $u\in H_{\textup{loc}}^{1}(M)$.

On a weighted Finsler manifold $(M, F, \fm)$, we have the Bochner formula  and its integrated version holds (see \cite{OS3}).
\begin{proposition}[Bochner formula]
For $u \in C^{\infty}(M)$, we have
\begin{equation}
\Delta^{\nabla u} \Big[ \frac{F^{2}(\nabla u)}{2} \Big]-d(\Delta u)(\nabla u) = \textup{Ric}_{\infty}(\nabla u) + ||\nabla^{2} u||^{2}_{\textup{HS}(\nabla u)}
\end{equation}
on $M_u$.
\end{proposition}

\begin{proposition}[Integrated Bochner formula]
For $u \in H^{2}_{\textup{loc}}(M)\cap \mathcal{C}^{1}(M)$ such that $\Delta u \in H^{1}_{\textup{loc}}(M)$ and all bounded test function $\phi \in H^{1}_{c}(M)\cap L^{\infty}(M)$, we have
\begin{equation}\label{eq:intB}
-\int_{M} d\phi \Big( \nabla^{\nabla u} \Big[ \frac{F^{2}(\nabla u)}{2} \Big]\Big) d\fm= \int_{M} \phi \big( d(\Delta u)(\nabla u)  + \textup{Ric}_{\infty}(\nabla u) + ||\nabla^{2}  u||^{2}_{\textup{HS}(\nabla u)}\big)d\fm.
\end{equation}
\end{proposition}
Here $||\cdot||_{\textup{HS}}$ stands for the Hilbert–Schmidt norm. We note that the regularities $u \in H^{2}_{\textup{loc}}(M)\cap \mathcal{C}^{1}(M)$ and $\Delta u \in H^{1}_{\textup{loc}}(M)$ are fulfilled by solutions of the heat equation (see \cite{OS1}).

As a consequence of the curvature-dimension condition, we obtain the following functional inequalities (see \cite{Oint, Obook}).
\begin{theorem}[Logarithmic Sobolev inequality]
Assume that $(M,F,\fm)$ is forward and backward complete with $\fm(M) = 1$ and $\textup{Ric}_{\infty} \geq K > 0$. For any locally Lipschitz, nonnegative function $\rho$ with $\int_{M}\rho d\fm = 1$, we have:
\begin{equation}
\int_{M}\rho \log{\rho} d\fm \leq \frac{1}{2K} \int_{M}\frac{F^{2}(\nabla\rho)}{\rho}d\fm,
\end{equation}
where the integrand on the right-hand side is regarded as $0$ on $\rho^{-1}(0)$.
\end{theorem}

\begin{theorem}[Poincar\'e inequality]\label{th:Poin}
Assume that $(M,F,\fm)$ is forward and backward complete with $\fm(M) = 1$ and $\textup{Ric}_{\infty} \geq K > 0$. For any locally Lipschitz function $u\in H_{0}^{1}(M)$, we have
\begin{equation}\label{eq:Poin}
\textup{Var}_{\fm}(u) := \int_{M} u^2 d\fm - \Big( \int_{M}u d\fm \Big)^2 \leq \frac{1}{K} \int_{M}{F^{2}(\nabla u)}d\fm.
\end{equation}
\end{theorem}
Poincar\'e inequality implies the spectral gap on Finsler manifolds.
\begin{theorem}[Spectral gap]\label{th:spec}
Assume that $(M,F,\fm)$ is forward and backward complete with  $\textup{Ric}_{\infty} \geq K > 0$ and $\fm(M) = 1$. We denote by $\lambda_{1}$ the first nonzero eigenvalue of the nonlinear Laplacian $-\Delta$. Then: \begin{equation}\label{eq:spec}
\lambda_{1} \geq K.
\end{equation}
\end{theorem}
Precisely, we say that $\lambda > 0$ is an eigenvalue of $-\Delta$ if there is a nonconstant function $u\in H_{0}^{1}(M)$ with $\int_{m}u d\fm = 0$ such that $\Delta u + \lambda u = 0$. These inequalities above is sharp as equality holds on the $1$-dimensional Gaussian space. It is clear that the equality of \eqref{eq:spec} implies the equality of \eqref{eq:Poin}. The reverse direction could be obtained by the following lemma.
\begin{lemma}\label{lem:Poin}
Let $H(u):= \frac{1}{K}\int_{M}{F^{2}(\nabla u)}d\fm - \textup{Var}_{\fm}(u)$. If a function $u\in H_{0}^{1}(M)$ with $\int_{M} u d\fm = 0$ satisfies $H(u) = 0$, then $\Delta u = -Ku$ in weak sense.
\end{lemma}
\begin{proof}
Given $\phi\in H_{0}^{1}(M)$, we have
\begin{align*}
    \frac{1}{2}\frac{d}{dt} \big[H(u+t\phi)\big]_{t=0} &= \frac{1}{K}\int_{M} d\phi(\nabla u)d \fm - \frac{1}{2}\frac{d}{dt} \big[\textup{Var}_{\fm}(u+t\phi)\big]_{t=0} \\
    &= \frac{1}{K}\int_{M} d\phi(\nabla u)d \fm - \int_{M} u\phi d\fm,
\end{align*}
where the first equality follows from Lemma 11.9 in \cite{Obook}. Since $u$ is a minimizer of $H(u)$, $\int_{M} d\phi(\nabla u)d \fm = K \int_{M} u \phi d\fm$ necessarily holds for all $\phi\in H_{0}^{1}(M)$. Hence, $\Delta u = -Ku$ in weak sense.
$\qedd$
\end{proof}

On Finsler manifolds, the \emph{Minkowski exterior content} is defined by
\[ \fm^{+}(A) := \liminf_{\ve \to 0} \frac{\fm(B^{+}(A,\ve) \setminus A)}{\ve} \]
for a Borel set $A$, where $B^{+}(A,\ve) := \{y\in M | \inf_{x\in A} d(x,y) < \ve \}$ denote the forward open $\ve$-neighborhood of $A$. Assuming $\fm(M) = 1$, we define the \emph{isoperimetric profile} by
 \[ \cI_{(M,F,\fm)}(\theta) := \inf \{ \fm^{+}(A) \,|\, A \subset M,\, \fm(A)=\theta \} \]
for $\theta \in (0,1)$, where $A$ runs over all Borel sets with $\fm(A)=\theta$. The isoperimetric profile of a $\textup{CD}(K,N)$-Finsler manifold is bounded from below by a functional as stated in the next theorem (\cite{Oneedle}).

\begin{theorem}[Isoperimetric inequality]
Let $(M,F,\fm)$ be a forward and backward complete weighted Finsler manifold of dimension $n\geq2$ and $\Lambda_{F} < \infty$. Assume $\fm(M)=1$, $\textup{Ric}_{N} \geq K$ and $\textup{diam}(M) \leq D$ for some $K \in \mathbb{R}$, $N\in (-\infty,0]\cap[n,\infty)$ and $D\in(0,\infty]$. Then we have
\[ \cI_{(M,F,\fm)}(\theta) \geq \Lambda_{F}^{-1}\cdot\cI_{K,N,D}(\theta),\]
where $\cI_{K,N,D}$ is a functional depending only on $K,N,D$.
\end{theorem}
For precise formula of $\cI_{K,N,D}$, please refer to \cite{Mi1,Mi2}. Here, we give the definition of $\cI_{K,N,D}$ on the case $K>0$, $N = \infty $ and $D = \infty$ (Bakry–Ledoux isoperimetric inequality), where the model space is the Gaussian space:
\[ \cI_{K,\infty,\infty}(\theta) =\sqrt{\frac{K}{2\pi}} \e^{-Ka_{\theta}^2/2}, \qquad
 \text{where}\ \ \sqrt{\frac{K}{2\pi}} \int_{-\infty}^{a_{\theta}} \e^{-Kt^2/2} \,dt =\theta. \]

On weighted Riemannian manifolds, the equivalence of the equality cases of the sharp spectral gap with logarithmic Sobolev (\cite{OT}) and isoperimetric inequalities (\cite{Ma2}) was shown via a localization technique called \emph{needle decomposition} (\cite{Kl}). We will discuss this phenomenon on reversible Finsler manifolds ($\Lambda_{F} = 1$) in Section 4.

\section{Rigidity of the spectral gap}

\subsection{Diffeomorphic splitting}
In this subsection, we investigate the rigidity problem of the sharp spectral gap on weighted Finsler manifolds of $\textup{Ric}_{\infty} \geq K > 0$. Our approach would follow the proof of Poincar\'e inequality via the Bochner formula. We add a regularity condition of the eigenfunction to avoid technical problems arising from the non-compactness.

\begin{theorem}[Diffeomorphic splitting]\label{th:main}
Let $ (M, F, \fm)$ be a complete weighted Finsler manifold of dimension $n\geq 2$ satisfying $\Lambda_{F} < \infty$, $\textup{Ric}_{\infty} \geq K > 0$ and $\fm(M) = 1$. Assume that there exists a nonconstant function $u \in H_{0}^{1}(M)$ with $\int_{M} u d\fm = 0$ such that $\Delta u + Ku = 0$ in weak sense. We assume furthermore
\begin{equation}\label{eq:technical}
    F(\nabla^{\nabla u}[F(\nabla u)]) \in L^{2}(M).
\end{equation}
Then $(M,g_{\nabla u})$ is isometric to the product space $\Sigma\times\mathbb{R}$ with $\Sigma = u^{-1}(0)$. Moreover, the measure $\fm$ splits as $\fm_{\Sigma}\otimes \gamma_{K}$ with $\gamma_{K}(t) = \textup{e}^{-Kt^2/2}dt$ is the Gaussian measure on $\mathbb{R}$.
\end{theorem}
\begin{proof}
We note that $e^{-Kt} u$ is a solution to the heat equation. Thus $u$ enjoys the regularities  $u \in H^{2}_{\textup{loc}}(M)\cap \mathcal{C}^{1}(M)$ and $\Delta u \in H^{1}_{\textup{loc}}(M)$ (see \cite{OS1}).
By an argument in Theorem 1.3 in \cite{CS} based on the improved Bochner inequality, $M$ must be non-compact. 
Given a fixed point $p\in M$, let $B^{-}_{R} := \{y\in M | d(y,p) < R \}$ denote the backward ball of radius $R>0$. Put $\phi_{R}(x) := \max \{1-d(x,B^{-}_{R}),0 \}$ where we set $d(x,B^{-}_{R}) := \inf_{y\in B^{-}_{R}} d(x,y)$, we have $\phi_{R}(x) = 1$ for $x\in B^{-}_{R}$ and $\phi_{R}(x) = 0$ for $x\in M\setminus B^{-}_{R+1}$. Moreover, we also have $F(\nabla \phi_{R}) \leq 1$ almost everywhere on $B^{-}_{R+1}\setminus B^{-}_{R}$. Observe that $\phi_{R}\in H_{c}^{1}(M)$, \eqref{eq:intB} and the weighted Ricci curvature bound imply
\begin{align*}\label{eq:BochCom}
-\int_{M} d\phi_{R} \Big( \nabla^{\nabla u} \Big[ \frac{F^{2}(\nabla u)}{2} \Big]\Big) d\fm&= \int_{M} \phi_{R} \big( d(\Delta u)(\nabla u)  + \textup{Ric}_{\infty}(\nabla u) + ||\nabla^{2}  u||^{2}_{\textup{HS}(\nabla u)}\big)d\fm \\
&\geq \int_{M} \phi_{R} \big( d(\Delta u)(\nabla u)  + KF^{2}(\nabla u) + ||\nabla^{2}  u||^{2}_{\textup{HS}(\nabla u)}\big)d\fm \\
&\geq  \int_{M} \phi_{R} \big( d(\Delta u)(\nabla u)  + KF^{2}(\nabla u)\big)d\fm. 
\end{align*}
We note that $0\leq \phi_{R}\leq 1$ and $\lim_{R\rightarrow\infty}\phi_{R}(x)=1$ point-wise on $M_{u}$ and all the integrands above vanish almost everywhere on $M\setminus M_{u}$ thanks to Lemma 12.12 in \cite{Obook}. The right-hand side converges to 
\[\int_{M} \big( d(\Delta u)(\nabla u)  + KF^{2}(\nabla u)\big)d\fm \]
as $R\rightarrow \infty$.
Concerning the left-hand side, we have
\begin{align*}
    \Bigg| -\int_{M}d\phi_{R}\Big( \nabla^{\nabla u} \Big[ \frac{F^{2}(\nabla u)}{2} \Big]\Big)d\fm \Bigg| &\leq
     \Lambda_{F}\int_{M\setminus B^{-}_{R}}F^{*}(d\phi_{R})F\Big( \nabla^{\nabla u} \Big[ \frac{F^{2}(\nabla u)}{2} \Big]\Big)d\fm \\
     &\leq \Lambda_{F}\int_{M\setminus B^{-}_{R}}F\Big( \nabla^{\nabla u} \Big[ \frac{F^{2}(\nabla u)}{2} \Big]\Big)d\fm
     \\
     &\leq \Lambda_{F}\int_{M\setminus B^{-}_{R}} F(\nabla u) F(\nabla^{\nabla u}[F(\nabla u)])d\fm \\
     &\leq  \Lambda_{F}\bigg(\int_{M\setminus B^{-}_{R}} F^{2}(\nabla u) d\fm\bigg)^{1/2} \bigg(\int_{M\setminus B^{-}_{R}} F^{2}(\nabla^{\nabla u}[F(\nabla u)])d\fm \bigg)^{1/2}
\end{align*}
tends to $0$ as $R\rightarrow \infty$ by \eqref{eq:technical} and $M\setminus B^{-}_{R}$ decreases to null set. We derived the last inequality by Holder inequality. We obtain
\begin{align*}
    0 \geq \int_{M} \big( d(\Delta u)(\nabla u)  + KF^{2}(\nabla u)\big)d\fm &= \int_{M} d(\Delta u)(\nabla u)d\fm + K \int_{M}F^{2}(\nabla u)d\fm \\ &= -\int_{M} (\Delta u)^{2}d\fm + K \int_{M}F^{2}(\nabla u)d\fm  \\&= -K^2 \int_{M} u^2 d\fm + K\int_{M} F^{2}(\nabla u) d\fm \\ &\geq 0.
\end{align*}
Therefore, we have the followings:
\begin{enumerate}[(i)]
    \item $\nabla^2 u = 0$ on $M_{u}$,
    \item $\textup{Ric}_{\infty}(\nabla u) = K F^{2}(\nabla u).$
\end{enumerate}
Let $\eta$ be any $\mathcal{C}^{1}$-curve in $M_{u}$. We calculate the derivative of $F^2(\nabla u)$ along $\eta$ thanks to Lemma 4.8 in \cite{Obook}:
\begin{align*}
   \frac{d}{dt}[F^2(\nabla u)] &= \frac{d}{dt}[g_{\nabla u}(\nabla u,\nabla u)]\\ &= g_{\nabla u}\big(D_{\dot{\eta}}^{\nabla u}(\nabla u),\nabla u\big) + g_{\nabla u}\big(\nabla u,D_{\dot{\eta}}^{\nabla u}(\nabla u)\big) \\
   &= 0,
\end{align*}
where in the last equality we deduce from (i). By the positivity and continuity of $F(\nabla u)$, we obtain that $F(\nabla u)$ is constant on $M$. Hence $\nabla u$ is $\mathcal{C}^{\infty}$ and all integral curves of $\nabla u$ are geodesic. The latter implies $\textup{Ric}_{\infty}(\nabla u) = \textup{Ric}_{\infty}^{g_{\nabla u}}(\nabla u)$ (see Remark 9.12 in \cite{Obook}). Moreover, $\nabla u$ is a parallel (and hence Killing) vector field with respect to $ g_{\nabla u}$ by (i). Therefore, $(M,g_{\nabla u})$ splits isometrically to the product space $\Sigma\times\mathbb{R}$ with $\Sigma = u^{-1}(0)$ via one-parameter family of transformation $\varphi_{t}: M \rightarrow M$ for  $t\in\mathbb{R}$ with respect to $\nabla u$. We write $\varphi_{t}(x,s) = (x,s+t)$ in this product structure. For each $x_{0}\in \Sigma$, $\sigma(t):=(x_{0},t)$ is an integral curve of $\nabla u$.

Now we consider the splitting on the measure. Let $\Psi$ be the density function associated with $\nabla u$ as in Definition \ref{def:Ric}. Recall in the proof of Theorem 12.7 in \cite{Obook} that
\[ ||\nabla^{2} u||^{2}_{\textup{HS}(\nabla u)} \geq \frac{(\Delta u+d\Psi(\nabla u))^2}{n}.\]
Combining this with $\Delta u + Ku = 0$ we obtain $d\Psi(\nabla u) = Ku$. We obtain that $(\Psi\circ\sigma)''(t) = K$ and the measure $\fm$ splits as $\fm_{\Sigma}\otimes \gamma_{K}$, where $\gamma_{K}(t) = \textup{e}^{-Kt^2/2}dt$ is the Gaussian measure on $\mathbb{R}$ and $\fm_{\Sigma}(A):=\sqrt{K/(2\pi)} \fm(A \times \R)$ defined through the isometry $M = \Sigma\times\mathbb{R}$. Thus we have 
\[ (\Sigma\times\mathbb{R},\fm_{\Sigma}\otimes \gamma_{K})\ni(x_{0},t)\rightarrow\varphi_{t}(x_{0})\in (M,\fm)\]
is diffeomorphic and measure-preserving.
$\qedd$
\end{proof}

\begin{remark}
The diffeomorphic splitting phenomenon could be compared with the Cheeger--Gromoll splitting theorem in \cite{Osplit}. As we mentioned in the introduction, eigenfunction $u$ plays a similar role to the Busemann function in the proof. 
\end{remark}
\begin{remark}
The technical condition \eqref{eq:technical} is unnecessary for weighted Riemannian manifolds and \textup{RCD}-spaces. This condition also holds in compact Finsler manifolds by the $H^2$-regularity of the solution of the heat equation.
One might refer to Section $3.4$ in \cite{Oisop} for the discussion on non-compact regularity of the improved Bochner inequality. 
\end{remark}

\subsection{Isometric splitting in the Berwald case}
This subsection discusses the splitting phenomenon in the particular class of Berwald spaces by following the argument in Section 5 of \cite{Osplit}. Throughout this part, let $(M,F,\fm)$ be a complete Berwald space of $\textup{Ric}_{\infty} > K > 0$ and $\fm(M) = 1$. We consider the structure of the level set of the eigenfunction by the following lemma.

\begin{lemma}[Eigenfunction is affine]\label{lem:afin}
Let $u$ be the eigenfunction in Theorem \ref{th:main}. Then for any geodesic $\xi:[0,l]\rightarrow M$, we have $(u\circ\xi)''=0$. In particular, the level set $u^{-1}(t)$ is totally convex and geodesically complete for each $t\in\mathbb{R}$. 
\end{lemma}
\begin{proof}
We deduce from $\nabla^2 u = 0$ in Theorem \ref{th:main} and Lemma 4.8 in \cite{Obook} that
\begin{align*}
    (u\circ\xi)'' &= \big[ g_{\nabla u}(\nabla u \circ \xi,\dot{\xi})\big]' \\
    &= g_{\nabla u} \big( D_{\dot{\xi}}(\nabla u \circ \xi),\dot{\xi}\big) + g_{\nabla u} (\nabla u \circ \xi, D_{\dot{\xi}}\dot{\xi}) \\
    &= 0.
\end{align*}
If $u(\xi(s)) = u(\xi(t))$ for some $s \neq t$ then $u\circ\xi$ is constant. Therefore, $u^{-1}(t)$ is totally convex and geodesically complete.
$\qedd$
\end{proof}
Now we obtain the isometric splititng on the affine structure of Berwald spaces.
\begin{theorem}[Isometric splitting]\label{th:split}
Let $ (M, F, \fm)$ be a complete weighted Finsler manifold of Berwald type of dimension $n\geq 2$ satisfying $\textup{Ric}_{\infty} \geq K > 0$ and $\fm(M) = 1$. Assume that there exists a function $u$ as in Theorem \ref{th:main}. Then we have
\begin{enumerate}[(i)]
    \item For all $t\in\mathbb{R}$, $\varphi_{t}$ is an measure-preserving isometry such that $\varphi_{t}(\Sigma) = M_{t}$ where $M_{t}:=u^{-1}(t)$. Moreover, $M = \bigsqcup_{t\in\mathbb{R}}M_{t}$,
    \item Let $\rho : M \rightarrow \Sigma$ be the natural projection $\rho(\varphi_{t}(x)) := x $ for $(x, t) \in \Sigma \times \mathbb{R}$. Then a curve $\xi: \mathbb{R} \rightarrow M$ is geodesic if and only if the projections $\rho(\xi) : \mathbb{R} \rightarrow \Sigma$ and $u(\xi) : \mathbb{R} \rightarrow \R$ are geodesic.
    \item $(\Sigma,F|_{T\Sigma},\fm_{\Sigma})$ is again a Berwald space and satisfies the Ricci curvature bound  $\textup{Ric}^{\Sigma}_{\infty} \geq K $,
\end{enumerate}
where $\Sigma$, $\varphi_{t}$ and $\fm_{\Sigma}$ are taken from proof of Theorem \ref{th:main}.
\end{theorem}\label{th:ber}
\begin{proof}

(i) For all $v\in T_{x}M$, Theorem \ref{th:main} shows that $V(t):= d\varphi_{t}(v)$ is a parallel vector field with respect to $g_{\nabla u}$ along geodesic $\eta(t) := \varphi_{t}(x)$. Since all integral curves of $\nabla u$ are geodesic, Lemma 2.3 in \cite{Osplit} yields that $V$ is a parallel vector field with respect to $F$. Hence, $(d\varphi_{t})_{x}: T_{x}M \rightarrow T_{\varphi_{t}(x)}M$ is isometric.

(ii) Notice that $F(\nabla u)$ is constant by proof of Theorem \ref{th:main}. Given an open set $U_{0}\subset \Sigma$ with local coordinates $(x^{i})_{i=1}^{n-1}$, we consider the orthonormal coordinate $(x^{i})_{i=1}^{n}$ with respect to $g_{\nabla u}$ of $U_{0}\times \mathbb{R}\subset M$ such that $x^{n} = \frac{u}{F(\nabla u)}$ . Note that
\[ \frac{\partial g_{ij}}{\partial x^n}(\nabla u) = 0 \qquad \textup{for } i,j = 1,2,\ldots,n\]
by (i) and
\[ \frac{\partial g_{in}}{\partial x^j}(\nabla u) = 0 \qquad \textup{for } i,j = 1,2,\ldots,n\]
by $g_{\nabla u}(\nabla u, T\Sigma) = 0$ and $F(\nabla u)$ is constant. We find that $\Gamma^{i}_{jk}(\nabla u) =0$ unless $1\leq i,j,k \leq n-1$ by a similar calculation as in Proposition 5.2 in \cite{Osplit}. Then the geodesic equation on $M$ splits into geodesic equations on $\Sigma$ and $\mathbb{R}$ as

\[ D_{\dot{\xi}}^{\dot{\xi}} = \sum_{i=1}^{n-1} \Big\{\ddot{\xi}^{i} + \sum_{i,k=1}^{n-1}\Gamma^{i}_{jk}(\xi)\dot{\xi}^{j}\dot{\xi}^{k} \Big\}\frac{\partial}{\partial x^{i}}\Big|_{\xi} + \ddot{\xi}^{n}\frac{\partial}{\partial x^{n}}\Big|_{\xi}.\]

(iii) $(\Sigma,F|_{T\Sigma},\fm_{\Sigma})$ is a Berwald space by the total convexity of $\Sigma$ in Lemma \ref{lem:afin} (see Exercise 5.3.5 in \cite{BCS}). Now we estimate the curvature bound. Taking a nonzero vector $v_{0}\in T_{x_{0}}\Sigma\setminus{0}$, we extend it into a vector field $V_{0}$ on a neighborhood $U_{0}\subset\Sigma$ of $x_{0}$ such that all of its integral curves are geodesic. Let $V$ be the vector field on $U:=U_{0}\times\mathbb{R}$ given by $V(y_{0},t):=(V_{0}(y_{0}),0)$, 
all integral curves of $V$ are geodesic by (ii). By the Riemannian characterizations of curvatures (see Theorem 5.12 in \cite{Obook}), we have
\[ \textup{Ric}_{\infty}^{\Sigma}(v_{0}) = \textup{Ric}_{\infty}^{M}((v_{0},0)) \geq KF^{2}\big((v_{0},0)\big) = KF|_{\Sigma}^{2}(v_{0}),\]
where $(v_{0},0)\in T_{x_{0}}\Sigma\times T_{0}\mathbb{R} = T_{(x_{0},0)}M$ via the product structure $M = \Sigma\times\mathbb{R}$.
$\qedd$
\end{proof}

\section{Application of the main theorem}
This section discusses some applications of the main results in the previous section to other rigidity problems. Our main approach is to use the localization method called \emph{needle decomposition} introduced by Klartag \cite{Kl} on weighted manifolds and Ohta \cite{Oneedle} on non-reversible Finsler manifolds. The concept of needle decomposition is to reduce inequalities on a high-dimensional space to those on geodesics while preserving the weighted Ricci curvature bound. On weighted Riemannian manifolds, it was used as the primary approach on the rigidity proof of isoperimetric inequality (\cite{Ma2}) and logarithmic Sobolev inequality (\cite{OT}). In both works, the authors showed that the rigidity of these inequalities implies the rigidity of Poincar\'e inequality and hence implies a same splitting phenomenon in the rigidity of the sharp spectral gap. We would follow this observation to establish the Finsler version. 

Firstly,  we define transport rays associated with a $1$-Lipschitz function.

\begin{definition}[Transport rays]\label{df:T-ray}
Let $\varphi$ be a $1$-Lipschitz function on $M$.
We say that a geodesic $\eta:I\rightarrow M$ from a closed interval $I \subset M$ is a \emph{transport ray} associated with $\varphi$
if $\varphi(y)-\varphi(x) = d(x,y)$ holds for all $x,y \in \eta(I)$ and $\eta$ can not be extended to a longer geodesic satisfying this property.
\end{definition}

Now we recall the needle decomposition theorem on Finsler manifolds in \cite{Oneedle}. 

\begin{theorem}[Ohta]\label{th:Ndl}
Let $(M,F,\fm)$ be a forward and backward complete weighted Finsler manifold of dimension $n\geq 2$ satisfying $\Ric_N \ge K > 0$ for $N\in(-\infty,0]\cup[n,\infty]$ and $K\in\mathbb{R}$,
and take a function $f \in L^1(M)$ such that $\int_{M} f \,d\fm = 0$.
Then there exists a $1$-Lipschitz function $\varphi$ on $M$ (called the guiding function) and a decomposition $\{I_{\eta}\}_{\eta \in Q}$ of $M$ into transpot rays associated with $\varphi$,
a measure $\nu$ on the index set $Q$ and a family of probability measures $\{\fm_{\eta}\}_{\eta \in Q}$ on $M$
satisfying the following.
\begin{enumerate}[{\rm (i)}]
\item\label{Ndl0} For $\nu$-a.e. $\eta\in Q$. $\textup{supp}\mu_{\eta}\subset I_{\eta}$,
\item\label{Ndl1}
For any measurable function $\phi$ on $M$, we have 
\[\int_{M} \phi d\fm = \int_{Q} \bigg(\int_{I_{\eta}}\phi d\fm_{\eta} \bigg) d\nu(\eta).\]

\item\label{Ndl2}
If $I_{\eta}$ is not a singleton, then the weighted Ricci curvature of
$(I_{\eta},|\cdot|,\fm_{\eta})$ satisfies $\Ric_N \ge K$.

\item\label{Ndl3}
For $\nu$-almost every $\eta \in Q$, we have $\int_{I_{\eta}} f \,d\fm_{\eta} =0$.
\end{enumerate}
\end{theorem}

Now we use this theorem to connect the equality cases of logarithmic Sobolev and Bakry--Ledoux isoperimetric inequalities with the equality case of Poincar\'e inequality.

\begin{corollary}[Rigidity of logarithmic Sobolev inequality]\label{cor:log}
Let $ (M, F, \fm)$ be a complete, reversible, boundaryless weighted Finsler manifold of dimension $n\geq 2$ satisfying $\textup{Ric}_{\infty} \geq K > 0$, $\fm(M) = 1$. 
Assume that there is a locally Lipschitz, nonconstant nonnegative function $\rho$ with $\int_{M}\rho d\fm = 1$ such that
\[ \int_{M}\rho \log{\rho} d\fm = \frac{1}{2K} \int_{M}\frac{F^{2}(\nabla\rho)}{\rho}d\fm \]
and we assume furthermore the function $u:=\log{\rho}$ satisfying \eqref{eq:technical}.
Then we obtain the same diffeomorphic splitting of $(M,F,\fm)$ as in Theorem \ref{th:main}. If $(M, F, \fm)$ is a Berwald space, then we obtain the isometric splitting as in Theorem \ref{th:split}.
\end{corollary}
\begin{proof}
We will show that the function $u:=\log{\rho}$ satisfying the equality case of Poincar\'e inequality \eqref{eq:Poin} and hence the spectral gap \eqref{eq:spec} by Lemma \ref{lem:Poin}.

Put $f:=\rho-1$ and notice that $f\in L^{1}(M)$ and $\int_{M}f d\fm = 0$. Let $\varphi, Q, \nu, \{(I_{\eta},\fm_{\eta})\}_{\eta \in Q}$ be the elements of the needle decomposition associated with $f$ as in Theorem \eqref{th:Ndl}. \eqref{Ndl3} in Theorem \ref{th:Ndl} implies $\int_{I_{\eta}}f d\fm_{\eta} = 0$ for almost every $\eta\in Q$, which yields $\int_{I_{\eta}} \rho d\fm_{\eta} = 1$. Therefore \eqref{Ndl2} in Theorem \eqref{th:Ndl} and the $1$-dimensional logarithmic Sobolev inequality yields
\begin{equation}\label{eq:log1}
\int_{I_{\eta}} \rho \log \rho d\fm_{\eta} \leq \frac{1}{2K} \int_{I_{\eta}} \frac{F^{2}(\nabla^{\eta}\rho)}{\rho}d\fm_{\eta},
\end{equation}
where $\nabla^{\eta}\rho$ denotes the slope of $\rho$ along $I_{\eta}$. We need the reversible condition here  since $F(-\dot{\eta}) \neq 1$ in the non-reversible case. As $F(\nabla^{\eta}\rho)\leq F(\nabla \rho)$, taking integration of the estimate above yields the logarithmic Sobolev inequality on $M$. Hence the equality in the condition implies the equality in \eqref{eq:log1} and $F(\nabla^{\eta}\rho) = F(\nabla \rho)$ for almost every $\eta\in Q$. The $1$-dimensional analysis in \cite{OT} yields the equality of Poincar\'e inequality for almost all $\eta\in Q$:
\[ \int_{I_{\eta}}u^2 d\fm_{\eta} - \Big( \int_{I_{\eta}}u d\fm_{\eta}\Big)^2 = \frac{1}{K} \int_{I_{\eta}} F^2(\nabla^{\eta} u) d\fm_{\eta}.\]
Note that $F(\nabla^{\eta}\rho) = F(\nabla \rho)$ for almost every $\eta\in Q$. By \eqref{Ndl1} in Theorem \eqref{th:Ndl} and the Cauchy–Schwarz inequality, we have
\begin{align*}
    \int_{M}u^2 d\fm - \Big( \int_{M}u d\fm\Big)^2 &= \int_{Q}\int_{I_{\eta}}u^2 d\fm_{\eta} d\nu(\eta) - \Big( \int_{Q}\int_{I_{\eta}}u d\fm_{\eta}d\nu(\eta)\Big)^2 \\
    &\geq \int_{Q}\int_{I_{\eta}}u^2 d\fm_{\eta} d\nu(\eta) -  \int_{Q}\Big(\int_{I_{\eta}}u d\fm_{\eta}\Big)^2 d\nu(\eta)\\
    &=  \frac{1}{K} \int_Q\int_{I_{\eta}} F^2(\nabla^{\eta} u) d\fm_{\eta}d\nu(\eta) \\
    &=  \frac{1}{K}\int_{M} F^2(\nabla u) d\fm.
\end{align*}
Hence, $u$ provides equality of the Poincar\'e inequality on $M$.
$\qedd$
\end{proof}

\begin{corollary}[Rigidity of Bakry--Ledoux isoperimetric inequality]
Let $ (M, F, \fm)$ be a complete, reversible, boundaryless weighted Finsler manifold of dimension $n \geq 2$ satisfying $\textup{Ric}_{\infty} > K > 0$, $\fm(M) = 1$. Assume that $\cI_{(M,F,\fm)}(\theta) = \cI_{K,\infty,\infty}(\theta)$ for some $\theta \in (0,1)$ and the guiding function $\varphi$ of the needle decomposition associated with the function $f:=1_{A}-\theta$ satisfies \eqref{eq:technical}, where $A$ is an isoperimetric minimizer.
Then we obtain the same diffeomorphic splitting of $(M,F,\fm)$ as in Theorem \ref{th:main}. If $(M, F, \fm)$ is a Berwald space, then we obtain the isometric splitting as in Theorem \ref{th:split}.
\end{corollary}
\begin{proof}
Let $A$ be the isoperimetric minimizer such that the guiding function $\varphi$ of the needle decomposition associated with the function $f:=1_{A}-\theta$ satisfies \eqref{eq:technical} and $Q, \nu, \{(I_{\eta},\fm_{\eta})\}_{\eta \in Q}$ be other elements of the needle decomposition associated with $f$. We will show that $\varphi$ satisfies the equality case of Poincar\'e inequality. \eqref{Ndl3} in Theorem \ref{th:Ndl} implies $\int_{I_{\eta}}f d\fm_{\eta} = 0$ for almost every $\eta\in Q$ and hence $\fm_{\eta}(A_{\eta}) = \theta$, where we put $A_{\eta}:=A\cap I_{\eta}$. Therefore, \eqref{Ndl2} in Theorem \ref{th:Ndl} and the $1$-dimensional isoperimetric inequality yields $\fm_{\eta}^{+}(A_{\eta}) \geq \cI_{K,\infty,\infty}(\theta)$. We also need the reversible condition here as $F(-\dot{\eta}) \neq 1$ in the non-reversible case. Taking integration over the index set $Q$, we obtain the isoperimetric inequality in $M$:
\[ \fm^{+}(A) \geq \int_{Q} \fm_{\eta}^{+}(A_{\eta})d\nu(\eta) \geq \cI_{K,\infty,\infty}(\theta).\]
Hence, we have $\fm_{\eta}^{+}(A_{\eta}) = \cI_{K,\infty,\infty}(\theta)$ for almost every $\eta\in Q$. $1$-dimensional analysis in \cite{Ma2} implies that $\varphi$ satisfies the equality of Poincar\'e inequality:
\[ \int_{I_{\eta}}\varphi^2 d\fm_{\eta} - \Big( \int_{I_{\eta}}\varphi d\fm_{\eta}\Big)^2 = \frac{1}{K} \int_{I_{\eta}} F^2(\nabla^{\eta} \varphi) d\fm_{\eta},\]
where $\nabla^{\eta}\varphi$ denotes the slope of $\varphi$ along $I_{\eta}$. We note that $1=F(\nabla^{\eta} \varphi)$ since $\eta$ is a transport ray associated with $\varphi$ and $F(\nabla \varphi)\leq 1$ by the Lipschitz condition of $\varphi$ and the reversibility of $(M,F,\fm)$. So we have $1=F(\nabla^{\eta} \varphi)\leq F(\nabla \varphi)\leq 1$ and hence all the inequalities become equality. By an analogous calculation as in Corollary \ref{cor:log}, we obtain that $\varphi$ provides equality of the Poincar\'e inequality on $M$.
$\qedd$
\end{proof}
\begin{remark}
We add the reversible condition of $(M,F,\fm)$ on these corollaries to use symmetric analysis on $1$-dimensional needles. For non-reversible Finsler manifolds, we might need another technique rather than needle decomposition to study these rigidity problems.
\end{remark}

\begin{remark}
One might use the needle decomposition to work with quantitative problems (the stability of a geometric inequality when equality nearly holds). On weighted Riemannian manifolds, the quantitative Bakry--Ledoux isoperimetric inequality was studied on \cite{MO1,MO2} and the proof covered the case of reversible Finsler manifolds.
\end{remark}

\end{document}